\documentclass[a4paper, 10pt]{amsart}
\usepackage{amsmath}
\usepackage{amssymb, color, hyperref}
\usepackage{amsfonts}

\theoremstyle{plain}
\newtheorem{theorem}{\bf Theorem}[section]
\newtheorem{proposition}[theorem]{\bf Proposition}
\newtheorem{lemma}[theorem]{\bf Lemma}

\theoremstyle{definition}

\newtheorem{definition}[theorem]{\bf Definition}

\newcommand{\N}{\mathbb N}
\newcommand{\Z}{\mathbb Z}

\renewcommand{\t}{\, | \,}
\newcommand{\und}{\;\mbox{ and }\;}

\newcommand{\la}{\langle}
\newcommand{\ra}{\rangle}

 \DeclareMathOperator{\ord}{ord}

 \DeclareMathOperator{\supp}{supp}

\newcommand{\bdot}{\boldsymbol{\cdot}}

\numberwithin{equation}{section}

\begin{document}

\title{On an inverse problem of
	Erd\H os, Kleitman, and Lemke}


\author{Qinghai Zhong}

\address{Institute for Mathematics and Scientific Computing,  University of Graz, NAWI Graz \\
	Heinrichstra{\ss}e 36, 8010 Graz, Austria}
\email{qinghai.zhong@uni-graz.at}

\subjclass[2010]{20D60, 11B75, 11P70}
\keywords{product-one sequences, zero-sum sequences, cyclic groups,  Dihedral groups, Dicyclic groups}

\begin{abstract}Let $(G, 1_G)$ be a finite group and let $S=g_1\bdot \ldots\bdot g_{\ell}$ be a nonempty sequence over $G$. We say $S$ is a tiny product-one sequence if its terms can be ordered such that their product equals $1_G$ and $\sum_{i=1}^{\ell}\frac{1}{\ord(g_i)}\le 1$. Let $\mathsf {ti}(G)$ be the smallest integer $t$ such that every sequence $S$ over $G$ with $|S|\ge t$ has a tiny product-one subsequence. The direct problem is to obtain the exact value of $\mathsf {ti}(G)$, while the inverse problem is to characterize the structure of long sequences over $G$ which have no tiny product-one subsequences. In this paper, we consider the inverse problem for cyclic groups and we also study both direct and inverse problems for dihedral groups and dicyclic groups.

\end{abstract}

\maketitle

\bigskip
\section{Introduction} \label{1}
\bigskip

Let $G$ be a finite group. By a sequence $S$ over $G$, we mean a finite unordered sequence with terms from $G$, where the repetition of elements is allowed. We say that $S$ is a product-one sequence if its terms can be ordered so that their product equals the identity element of $G$. If $G$ is abelian, a product-one sequence is also called a zero-sum sequence.
Suppose  $S=g_1\bdot \ldots\bdot g_{\ell}$. If $1\le \ell\le \max\{\ord(g)\colon g\in G\}$, we say $S$ is a short product-one sequence. If $0<\sum_{i=1}^{\ell}\frac{1}{\ord(g_i)}\le 1$, we say $S$ is a tiny product-one sequence. Clearly, a tiny product-one sequence is a short product-one sequence.

 The small Davenport constant $\mathsf d (G)$ is the maximal integer $\ell$ such that there is a sequence of length $\ell$ which has no non-trivial product-one subsequences. We denote by $\eta (G)$ (or $\mathsf {ti} (G)$ respectively) the smallest integer $\ell$ such that every sequence of length at least $\ell$ has a short (or tiny respectively) product-one subsequence. Then $\mathsf d(G)+1\le \eta(G)\le \mathsf {ti}(G)$. The study of $\mathsf d(G)$ and $\eta(G)$ is a central topic in zero-sum theory which is a popular  branch of  combinatorial number theory. For more details, see \cite{Ga-Ge06b} for a survey and for recent progress, we refer to \cite{Ga-Ge-Sc07a, Sc10b, Ga-Ge-Gr10a, Sc11b, Ch-Sa14a, Gi18a, Gi-Sc19a, Gi-Sc19b}.

The investigations on $\mathsf {ti}(G)$ originate in the following conjecture, addressed by Erd\H os and
Lemke in the late eighties (see \cite[Introduction]{Le-Kl89}).

\smallskip
{\it Is it true that out of $n$ divisors of $n$, repetitions
being allowed, one can always find a certain number of them that sum up to $n$? }
\smallskip

Motivated by this
conjecture, Kleitman and Lemke \cite[Theorem 1]{Le-Kl89} proved the following stronger result.

\smallskip
{\it For any given integers $a_1, \ldots, a_n$,  there is a non-empty
subset $I \subset [1,n]$ such that $n$ divides $\sum_{i\in I}a_i$ and $\sum_{i\in I}\gcd(a_i, n)\le n$. With our notation, this is $\mathsf {ti}(C_n)\le n$ and it is easy to see that $\mathsf {ti}(C_n)=n$, where $C_n$ is a cyclic group of order $n$.
}
\smallskip

In the same paper, they also gave the following conjecture and confirmed the conjecture for elementary $p$-groups, dihedral groups,  dicyclic groups, and groups with $|G|\le 15$ (see \cite[Open Problems]{Le-Kl89}).

\smallskip
\noindent{\bf Conjecture A. }{\it 
Let $G$ be a finite group. Then $\mathsf {ti}(G)\le |G|$.
}
\smallskip

This conjecture for all finite abelian groups was proved by Geroldinger \cite{Ge93}.  An alternative
proof for all finite abelian groups,  using graph pebbling,  was later found by Elledge and Hurlbert \cite[Theorem 2]{El-Hu05}. For more work on applications of graph pebbling to zero-sum problems
we refer to \cite{ch89, De97, Hu99, Hu05}.
In 2012,  B. Girard \cite[Theorem 2.1]{Gi12},  using a result of Alon and Dubiner \cite{Al-Du95},  gave a new upper bound of $\mathsf {ti}(G)$ for all finite abelian groups $G$, which provided the right order of magnitude. The exact value of $\mathsf {ti}(G)$ is only known for limited groups. For finite abelian groups $G$ of  rank $2$,  B. Girard conjectured that $\mathsf {ti}(G)=\eta(G)$, which is still wide open, while there are finite abelian groups $G$ of rank $3$ such that $\mathsf {ti}(G)>\eta(G)$ (see \cite[Theorem 1.3]{Fa-Ga-Pe-Wa-Zh13}).

The inverse problem of $\mathsf {ti}(G)$ is to characterize the structure of long sequences $S$ (the extremal case is $|S|=\mathsf {ti}(G)-1$) which have no tiny product-one subsequences. In this manuscript, we give a first characterization of the inverse problem of $\mathsf {ti}(G)$ for  cyclic groups $G$, where we will use  {\it smooth sequences} (see Definition \ref{de1}).

\begin{theorem}	\label{th1}
	For every real $\epsilon>0$, there exists  $N_{\epsilon}\in \N$ such that for all cyclic groups $G$ with $|G|\ge N_{\epsilon}$ and for all sequences $S$ over $G$ with $|S|\ge (\frac{1}{2}+\epsilon)|G|$, if  $S$ has no tiny product-one subsequences, then  $S$ is smooth.
\end{theorem}

Most of intention is on abelian groups. Now we turn to  non-abelian groups. 
Dihedral groups and dicyclic groups are the most studied non-abelian groups in zero-sum theory (see \cite{Br-Ri18, Ba07, Ga-Lu08a, Oh-Zh19a, Oh-Zh19b}).
By the inverse result of cyclic groups in hand, 
we can study both the direct and  the inverse problems of $\mathsf {ti}(G)$ for dihedral and dicyclic groups $G$.

\begin{theorem}\label{th2} Let $G$ be a dihedral group of order $2n$ with $n\ge 3$. 
	\begin{enumerate}
		\item $\mathsf {ti}(G)=2n$.
		
		\item Let $S$ be a sequence over $G$ with $|S|=2n-1$. Then   $S$ has no tiny product-one subsequences if and only if   there exist $g,h\in G$ with $G=\langle g,h\colon g^n=h^2=1, hg=g^{-1}h\rangle$ such that $S=g^{[n-1]}\bdot h\bdot hg\bdot \ldots \bdot hg^{n-1}$.
	\end{enumerate}
\end{theorem}

\begin{theorem}\label{th3} Let $G$ be a dicyclic group of order $4n$ with $n\ge 116$.
		\begin{enumerate}
		\item $\mathsf {ti}(G)=2n+1$.
		
		\item Let $S$ be a sequence over $G$ with $|S|=2n$. Then  $S$ has no tiny product-one subsequences if and only if there exist $g,h\in G$ with $G=\langle g,h\colon  g^{2n}=1, h^2=g^n, \text{ and } hg=g^{-1}h\rangle$ such that $S=g^{[2n-1]}\bdot h$.
	\end{enumerate}
\end{theorem}

The manuscript is organized as following. In Section $2$, we gather all necessary notation and definitions.
 In Section $3$, we focus on the inverse problem for cyclic groups and give the proof of Theorem \ref{th1}. In Section $4$, we consider dihedral and dicyclic groups $G$ and the proofs of Theorems \ref{th2} and \ref{th3} will be provided. In addition,
we also study both the direct and the inverse problems of $\eta(G)$ (see Theorems \ref{t1} and \ref{t2}).


\bigskip
\section{Preliminaries} \label{2}
\bigskip

We denote by $\N$ the set of positive integers and let $\N_0 = \N \cup \{ 0 \}$ be the set of non-negative integers. 
For non-negative integers $a, b \in \N_0$, let $[a,b] = \{x \in \Z \mid a \le x \le b\}$ be the discrete interval. 

 Let $G$ be a multiplicatively written finite group with identity element $1_G$ and let $G_0 \subset G$ be a subset. For an element $g \in G$, we denote by $\ord(g) \in \N$ the order of $g$  and by $\la G_0 \ra \subset G$ the subgroup generated by $G_0$.

The elements of the free abelian monoid $\mathcal F (G)$ will be called  {\it sequences} over $G$.  This terminology goes back to combinatorial number theory. Indeed, a sequence over $G$ can be viewed as a finite unordered sequence with terms from $G$, where the repetition of elements is allowed. We briefly discuss our notation here which follows the monograph \cite[Chapter 10.1]{Gr13a}. In order to avoid confusion between multiplication in $\mathcal F(G)$ and multiplication in $G$, we denote multiplication in $\mathcal F (G)$ by the boldsymbol $\bdot$  and   use brackets for all exponentiation in $\mathcal F (G)$. Thus a sequence $S \in \mathcal F (G)$ can be written in the form
$$
S \, = \, g_1 \bdot \ldots \bdot g_{\ell} \, = \, {\small \prod}^{\bullet}_{g\in G} \, g^{\mathsf v_g(S)} \, \in \, \mathcal F (G),
$$
where $g_1, \ldots, g_{\ell} \in G$ are the terms of $S$. For $g \in G$,
\begin{itemize}
	\item $\mathsf v_g (S) = | \{ i \in [1,\ell] \colon  g_i = g \} |$ denotes the {\it multiplicity} of $g$ in $S$,
	
	\smallskip
	\item $\supp(S) = \{ g \in G \colon \mathsf v_{g} (S) > 0 \}$ denotes the {\it support} of $S$,
	
	\smallskip
	\item $\mathsf h (S) = \max \{ \mathsf v_g (S) \colon g \in G \}$ denotes the {\it height} of $S$,

	\smallskip
	\item $|S|=\ell=\sum_{g\in G}\mathsf v_g(G)$ denotes the {\it length} of $S$, and 
	
	\smallskip
	\item $\mathsf k(S)=\sum_{i=1}^{\ell}\frac{1}{\ord(g_i))}$ denotes the {\it cross number } of $S$.
	
\end{itemize}
A {\it subsequence} $T$ of $S$ is a divisor of $S$ in $\mathcal F (G)$ and we write $T \mid S$. A {\it proper subsequence} $T$ of $S$ is a sequence of $S$ with $T\neq S$.
 Let $G_0 \subset G$ be a subset. we denote by $S_{G_0}$ the subsequence of $S$ consisting of all terms from $G_0$ and if $T \t S$, we let $S \bdot T^{[-1]}$ be the subsequence of $S$ obtained by removing the terms of $T$ from $S$.

Let $S=g_1\bdot \ldots \bdot g_{\ell} \in \mathcal F (G)$, where $\ell\in \N$ and $g_1,\ldots, g_{\ell}\in G$. We denote by
$$
\pi(S) \, = \, \{ g_{\tau (1)} \ldots  g_{\tau (\ell)} \in G \colon  \tau \mbox{ a permutation of $[1, \ell]$} \} \, \subset \, G \quad \und \quad \Pi (S) \, = \, {\bigcup_{T\neq 1_{\mathcal F(G)}\text{ and }T \t S}} \pi (T) \, \subset \, G \,,
$$
the {\it set of products} and {\it subsequence products} of $S$. In particular, if $\pi(S)$ has only one element (for example, $G$ is abelian), we denote the singleton by $\sigma(S)$. 

 Note that if $S = 1_{\mathcal F (G)}$ is empty, we use the convention  $\pi (S) = \{ 1_G \}$. The sequence $S$ is called
\begin{itemize}
	\item a {\it product-one sequence} if $1_G \in \pi (S)$,
	
	\smallskip
	\item {\it product-one free} if $1_G \notin \Pi (S)$,
	
	\smallskip
	\item a {\it minimal product-one free} if $S$ is nonempty,  $1_G \in \pi (S)$, and $1_G\not\in \pi(T)$ for any nonempty proper subsequence $T$ of $S$, and
	
	\smallskip
	\item {\it square-free} if $\mathsf h (S) \le 1$.
\end{itemize}
If $S = g_1 \bdot \ldots \bdot g_{\ell}$ is a product-one sequence with $1_G = g_1 \ldots g_{\ell}$, then $1_G = g_i \ldots g_{\ell}g_1 \ldots g_{i-1}$ for every $i \in [1, \ell]$.
Every group homomorphism $\theta : G \rightarrow H$ can extend to a monoid homomorphism $\theta : \mathcal F (G) \rightarrow \mathcal F (H)$, where $\theta (S) = \theta (g_1)\bdot \ldots \bdot \theta (g_{\ell})$.
Then $\theta (S)$ is a product-one sequence if and only if $\pi(S) \cap \ker(\theta) \neq \emptyset$.

\begin{definition}
	Let $G$ be a finite group. We define 
	\begin{itemize}
		\item the small Davenport constant of $G$ 
		\[
		\mathsf d (G) \, = \, \sup \big\{ |S| \colon S \in \mathcal F (G) \mbox{ is product-one free } \big\}\,,
		\]
		
		\item the small cross number of $G$
		\[
		\mathsf k (G) \, = \, \sup \big\{ \mathsf k(S) \colon S \in \mathcal F (G) \mbox{ is product-one free } \big\} \,,
		\]
		
		\item the  cross number of $G$
		\[
		\mathsf K (G) \, = \, \sup \big\{ \mathsf k(S) \colon S \in \mathcal F (G) \mbox{ is a minmial product-one sequence } \big\} \,.
		\]
		\end{itemize}
	
\end{definition}

Note that the definition of $\mathsf d(G)$ here is equivalent to the definition in Section $1$ and the cross numbers are a crucial tool in the study of combinatorial factorization theory. It is easy to see that $\mathsf K(G)>\mathsf k(G)$.
For more on the small Davenport constant, on the small cross number, on the differences between the (large) Davenport
constant and the small davenport constant, and on the differences between the (large) cross number and the small cross number, we refer to \cite{Ge-HK06a, Ge09a, Cz-Do-Sz18, Cz-Do-Ge16a, Ge-Gr13a,  Gr13b}.

\bigskip
\section{On cyclic groups}
\bigskip

In this section, we deal with cyclic groups. Note that $\mathsf {ti}(G)=|G|$ if $G$ is a finite cyclic group. We will use this result many times without further mention. To characterize the structure of sequences without tiny product-one subsequences, we need the following definition.
%
%

\begin{definition}\label{de1}
	Let $G$ be a finite cyclic group of order $n$ and let $S$ be a nonempty sequence over $G$. 
	For every $g\in G$ with $\ord(g)=n$, we have $S=g^{n_1}\bdot g^{n_2}\bdot \ldots\bdot g^{n_{\ell}} $, where $\ell\in \N$, and  $n_1,\ldots, n_{\ell}\in [1, n]$.	
 We say $S$ is {\it $g$-smooth} if $n_1+\ldots +n_{\ell}<n$ and $\Pi(S)=\{g, g^2, \ldots, g^{n_1+\ldots+n_{\ell}}\}$. 
 We say $S$ is {\it smooth} if there exists $g\in G$ with $\ord(g)=n$ such that $S$ is $g$-smooth. In particular, every smooth sequence is product-one free and every $g$-smooth sequence $|S|$ of length $n-1$ must be the form $S=g^{[n-1]}$. For more on smooth sequences, we refer to \cite[Section 5.2]{Ge09a}.
		\end{definition}

%
%
%

\begin{lemma}\label{le1}
	Let $G$ be a finite cyclic group of order $n\ge 3$ and let $S=g^{n_1}\bdot \ldots\bdot g^{n_{\ell}}$ be a sequence over $G$, where $\ell\in \N$, $n_1,\ldots, n_{\ell}\in [1,n]$, and $g\in G$ with $\ord(g)=n$.
	\begin{enumerate}
		\item $\mathsf k(S)\le \frac{n_1+\ldots+n_{\ell}}{n}$. In particular, if $S$ is smooth, then $\mathsf k(S)<1$.
		
		\item If $S$ is a smooth sequence over $G$ and there exists $h\in G$ such that $S\bdot h$ is not product-one free,  then $S\bdot h$ has a tiny product-one subsequence.
	\end{enumerate}
\end{lemma}

\begin{proof}
	1. 
	For every $i\in [1,\ell]$, we have $\ord(g^{n_i})=\frac{n}{\gcd(n_i,n)}$ and hence $\frac{1}{\ord(g^{n_i})}=\frac{\gcd(n_i,n)}{n}\le \frac{n_i}{n}$. Therefore $\mathsf k(S)=\sum_{i=1}^{\ell}\frac{1}{\ord(g^{n_i})}\le \frac{n_1+\ldots +n_{\ell}}{n}$. In particular, if $S$ is smooth, to simplify the notation, we may assume that $S$ is $g$-smooth and hence $n_1+\ldots+n_{\ell}<n$. Then the assertion follows.
	
2. 	To simplify the notation, we may assume $S$ is $g$-smooth. There exists $m\in [1,n]$ such that $h=g^m$.
 Since $S$ is product-one free, there exists a subset $I\subset [1,\ell]$ such that $g^m\bdot \prod_{i\in I}^{\bullet}g^{n_i}$ is a minimal product-one sequence, whence $m+\sum_{i\in I}n_i=n$. In view of $\mathsf k(g^m\bdot \prod_{i\in I}^{\bullet}g^{n_i})\le \frac{m}{n}+\sum_{i\in I}\frac{n_i}{n}=1$, we obtain $S\bdot h$ has a tiny product-one subsequence.
\end{proof}

\begin{theorem}\label{t3.2}
	Let $G$ be a cyclic group with order $n\ge 3$ and  
	let $S\in\mathcal F(G)$ with $|S|\ge \frac{n+1}{2}$. If $S$ is product-one free, then  $S$ is smooth. In particular, if $|S|=n-1$, then $S=g^{[n-1]}$ for some $g\in G$ with $\ord(g)=n$.
	\end{theorem}
\begin{proof}
This was proved independently by Savchev-Chen \cite{Sa-Ch07} and by Yuan \cite{Yu07}. A proof  is also provided in the book \cite[Theorem 5.1.8]{Ge09a}.
\end{proof}

\begin{lemma}\label{l3.3}
	Let $G$ be a cyclic group with order $n=p^{\ell}\ge 3$ is a prime power.
	
	\begin{enumerate}
		\item $\mathsf K(G)=1$ and hence every minimal product-one sequence is a tiny product-one sequence.
		
		\item 	Let $S\in\mathcal F(G)$ with $|S|\ge \frac{n+1}{2}$. If $S$ has no tiny product-one sequence, then  $S$ is smooth.

		\item Let $S\in \mathcal F(G)$ with $\mathsf k(S)\ge t$, where $t\in \N$. Then $S$ has at least $t$ pairwise disjoint tiny product-one subsequences.
	\end{enumerate}

\end{lemma}

\begin{proof}
1. follows from   \cite[Theorem 5.1.14]{Ge09a} and the definition of tiny product-one sequences.

2. If $S$ has no tiny product-one subsequences, then by 1. $S$ has no product-one subsequences. The assertion follows by Theorem \ref{t3.2}.	

3. Suppose $S=T_1\bdot T_2\bdot \ldots\bdot T_s\bdot T'$, where $T_i, i\in [1,s]$ are tiny product-one subsequences of $S$ and $T'$ is a subsequence that has no tiny product-one subsequences. It follows by 1. that $\mathsf k(T_i)\le 1$ for all $i\in [1,s]$ and $T'$ is product-one free, whence $\mathsf k(S)\le s+\mathsf k(T')<s+1$. The assertion follows by $\mathsf k(S)\ge t$.	
	\end{proof}

\begin{theorem}\label{t3.1}
	Let $G$ be a cyclic group of order $n=mq^r\ge 3$, where $m,r\in \N$ and $q$ is a prime with $\gcd(m,q)=1$,  and let $S$ be a sequence of length 
	$|S|\ge 
	\frac{n+1}{2}+(r+1)(m-1)=\frac{n+1}{2}+\frac{r+1}{q^r}n-(r+1)$.
	If $S$ has no tiny product-one subsequences, then $S$ is smooth.
\end{theorem}

\begin{proof}
	Suppose $m=q_1^{r_1}\ldots q_t^{r_t}$, where $t\in \N_0$, $r_1,\ldots, r_t\in \N$, and $q_1,\ldots, q_t$ are pairwise distinct primes. If $t=0$, then $m=1$ and hence the assertion follows by Lemma \ref{l3.3}.2.

	Suppose $t\ge 1$ and
	suppose $G\cong H_1\oplus H_2$, where $H_1, H_2$ are  cyclic subgroups  with $|H_1|=m$ and $|H_2|=q^{r}$. Let $\phi_1\colon G\rightarrow H_1$ and $\phi_2\colon G\rightarrow H_2$ be projections from $G$ to $H_1$ and $H_2$. Then the sequence $S$ has a decomposition $$S=T_0\bdot S_0\bdot S_1\bdot \ldots\bdot S_{r}\,,$$
	where $S_0$ consists by terms $g$ with $\phi_1(g)=0$, $T_0$ consists by terms $g$ with $\phi_2(g)=0$, and $S_i$, $1\le i\le r$, consists by terms $g$ with $\ord(\phi_2(g))=q^{i}$. For each $i\in [1, r]$ we choose all pairwise disjoint tiny product-one subsequences $\phi_1(W_{i,1}), \ldots, \phi_1(W_{i,k_i})$ of $\phi_1(S_i)$. Then $\phi_1(T_i)=\phi_1(S_i\bdot (W_{i,1}\bdot \ldots W_{i,k_i})^{[-1]})$ has no tiny product-one subsequences, whence $|T_i|=|\phi_1(T_i)|\le m-1$. Since $T_0$ is a sequence over $H_1$ and has no tiny product-one subsequence, we obtain $|T_0|\le m-1$. Let 
	\begin{align*}
	W&=S_0\bdot W_{1,1}\bdot \ldots\bdot  W_{1,k_1}\bdot \ldots \bdot W_{r,1}\bdot \ldots \bdot W_{r, k_{r}}\in \mathcal B(G)\,,\\
\text{and} \quad	W'&=S_0\bdot \sigma(W_{1,1})\bdot\ldots\bdot \sigma(W_{1,k_1})\bdot \ldots \bdot \sigma(W_{r,1})\bdot \ldots \bdot \sigma(W_{r, k_{r}})\in \mathcal B(H_2)\,.
	\end{align*}
	Then $|W|\ge |S|-(r+1)(m-1)\ge \frac{n+1}{2}$	and 
	$|W'|\ge \left\lceil\frac{|W|}{m}\right\rceil\ge \frac{q^r+1}{2}$. Note that 
 for each $(i,j)\in \big\{(i,j_i)\colon i\in [1, r] \text{ and } j_i\in [1, k_i]\big\}$, we have $$\mathsf k(W_{i,j})=\frac{\mathsf k(\phi_1(W_{i,j}))}{q^i}\le \frac{1}{q^i}\le \frac{1}{\sigma(W_{i,j})}=\mathsf k(\sigma(W_{i,j}))\,.$$
	
	If $W'$ has a minimal product-one subsequnce  $W_0'$, it follows by Lemma \ref{l3.3}.1 that $W_0'$ is a tiny product-one subseqeunce over $H_2$, say $W_0'=V_0\bdot \prod_{(i,j)\in I}^{\bullet} \sigma(W_{i,j})$, where $V_0$ is a subsequence of $S_0$ and $I\subset \big\{(i,j_i)\colon i\in [1, r] \text{ and } j_i\in [1, k_i]\big\}$. 
	Then  $W_0=S_0\bdot \prod_{(i,j)\in I}^{\bullet} W_{i,j}$ is a  product-one subsequence of $S$ and $\mathsf k(W_0)\le \mathsf k(W_0')\le 1$,
	a contradiction.
	
	If $W'$ is product-one free, it follows by Theorem \ref{t3.2} that  $W'$ is smooth over $H_2$ and hence by Lemma \ref{le1} $\mathsf k(W')< 1$. Therefore $\mathsf k(W)\le \mathsf k(W')<1$. If $W$ has a nonempty product-one subsequence $W_0$, then $\mathsf k(W_0)\le \mathsf k(W)< 1$ and hence $W_0$ is a tiny product-one subsequence, a contradiction . If $W$ is product-one free, let $U$ be the maximal product-one free subsequence of $S$ with $W\mid U$.
		It follows by Theorem \ref{t3.2} and $|U|\ge |W|\ge \frac{n+1}{2}$ that $U$ is smooth over $G$.  Assume to the contrary that $U\neq S$. Let $h\in \supp(S\bdot U^{[-1]})$. Then $U\bdot h$ is not product-one free and it follows by Lemma \ref{le1} that $U\bdot h$ has a tiny product-one subsequence, a contradiction.
\end{proof}

\begin{proof}[Proof of Theorem \ref{th1}]
		Suppose $n=|G|=q_1^{r_1}\ldots q_t^{r_t}$, where $t\in \N$, $r_1,\ldots, r_t\in \N$, and $q_1,\ldots, q_t$ are pairwise distinct primes.  Suppose $q^r=\max\{q_1^{r_1}, \ldots, q_t^{r_t}\}$ and $m=\frac{n}{q^r}$.
		Since $\lim_{n\rightarrow \infty}\frac{r+1}{q^r}=0$, for every real $\epsilon>0$, there exists $N_{\epsilon}\in \N$ such that for all $n\ge N_{\epsilon}$,  we have $\frac{r+1}{q^r}\le \epsilon$.
		Now the assertion follows by Theorem \ref{t3.1}.
\end{proof}

\medskip
\begin{proposition}\label{c3.2}
	Let $G$ be a cyclic group of order $n\ge 3$  and let $S$ be a sequence of length  $|S|\ge \frac{9n}{10}$.  If  $S$ has no tiny product-one subsequence, then $S$ is smooth.
\end{proposition}
\begin{proof} Suppose  $S$ has no tiny product-one subsequence and 
		suppose $n=q_1^{r_1}\ldots q_t^{r_t}$, where $t, r_1,\ldots, r_t\in \N$, and $q_1<q_2<\ldots<q_t$ are pairwise distinct primes. 
		 If $q_t\ge 5$, then $$|S|\ge \frac{9n}{10}=\frac{n+1}{2}+\frac{2}{5}n-\frac{1}{2}\ge \frac{n+1}{2}+\frac{r_t+1}{q_t^{r_t}}n-(r_t+1)\,.$$ 
	It follows by Theorem \ref{t3.1} that $S$ is smooth.
	
	Suppose $q_t\le 3$. Then $t\in [1,2]$. If $t=1$,  the assertion follows by Lemma \ref{l3.3}.2. Now we suppose $t=2$. Then $q_1=2$ and $q_2=3$.

	 If $r_2\ge 2$, then 
	 $$|S|\ge \frac{9n}{10}=\frac{n+1}{2}+\frac{2}{5}n-\frac{1}{2}\ge \frac{n+1}{2}+\frac{2+1}{3^{2}}n-\frac{1}{2}\ge \frac{n+1}{2}+\frac{r_2+1}{3^{r_2}}n-(r_2+1)\,.$$ 
	 It follows by Theorem \ref{t3.1} that $S$ is smooth.
	 
	 If $r_2=1$ and $r_1\ge 4$, then  $$|S|\ge \frac{9n}{10}=\frac{n+1}{2}+\frac{2}{5}n-\frac{1}{2}\ge \frac{n+1}{2}+\frac{4+1}{2^{4}}n-\frac{1}{2}\ge \frac{n+1}{2}+\frac{r_1+1}{2^{r_1}}n-(r_1+1)\,.$$ 
	 It follows by Theorem \ref{t3.1} that $S$ is smooth.
	 
	 If $r_2=1$ and $r_1=3$, then $n=24$ and $|S|\ge \frac{9n}{10}$ implies that $|S|\ge 22$. Since $$\frac{n+1}{2}+(r_1+1)(q_2^{r_2}-1)=25/2+4(3-1)=20.5\le 22\le |S|\,,$$ it follows by Theorem \ref{t3.1} that $S$ is smooth.
	
	  If $r_2=1$ and $r_1=1$, then $n=6$ and it follows by $|S|\ge \frac{9n}{10}$  that $|S|\ge 6=|G|$, whence $S$ has a tiny product-one subsequence, a contradiction.
	  
	  \smallskip
	  Now we deal with the only left case that   $r_2=1$ and $r_1=2$, which implies $n=12$. It follows by $|S|\ge \frac{9n}{10}$  that $|S|\ge 11$. Suppose $G\cong H_1\oplus H_2$, where $H_1, H_2$ are  cyclic subgroups  with $|H_1|=4$ and $|H_2|=3$. Let $\phi_1\colon G\rightarrow H_1$ and $\phi_2\colon G\rightarrow H_2$ be projections from $G$ to $H_1$ and $H_2$. Then the sequence $S$ has a decomposition $$S=T_0\bdot S_0\bdot S_1\,,$$
	   where $S_0$ consists by terms $g$ with $\phi_1(g)=0$, $T_0$ consists by terms $g$ with $\phi_2(g)=0$, and $S_1$ consists by terms $g$ with $\phi_1(g)\neq 0$ and $\phi_2(g)\neq 0$. 
	   If $|S_0|\ge 3=|H_2|$, then $S_0$ has a tiny product-one subsequence, a contradiction. If $|T_0|\ge 4=|H_1|$, then $T_0$ has a tiny product-one subsequence, a contradiction. Therefore $|T_0|\le 3$ and $|S_0|\le 2$. We distinguish three cases.
	   
	   \medskip
	   Suppose  $|S_0|=2$.  Then $|S_1|\ge 6\ge |H_1|$ and hence $\phi_1(S_1)$ has a tiny product-one subsequence $\phi_1(W_1)$, whence $\mathsf k(W_1)\le \frac{1}{3}\le \frac{1}{\ord(\sigma(W_1))}=\mathsf k(\sigma(W_1))$. Since $S_0\bdot \sigma(W_1)$ is a sequence over $H_2$ with length $|S_0\bdot\sigma(W_1)|=3\ge |H_2|$, we obtain $S_0\bdot \sigma(W_1)$ and hence $S_0\bdot W_1$ have a tiny product-one subsequence, a contradiction.

	    Suppose $|S_0|=1$.  Then $|S_1|\ge 7$. If $\mathsf k(\phi_1(S_1))\ge 2$, then by Lemma \ref{l3.3}.3 we can choose two disjoint tiny product-one subsequences $\phi(W_1)$, $\phi(W_2)$ of $\phi_1(S_1)$.  Since $\mathsf k(W_i)\le \mathsf k(\sigma(W_i))$ for each $i\in [1,2]$ and $S_0\bdot \sigma(W_1)\bdot \sigma(W_2)$ is a sequence over $H_2$ with length $|S_0\bdot \sigma(W_1)\bdot \sigma(W_2)|=3\ge |H_2|$, we obtain $S_0\bdot \sigma(W_1)\bdot \sigma(W_2)$ and hence $S_0\bdot W_1\bdot W_2$ have a tiny product-one subsequence, a contradiction.
	     Otherwise $\mathsf k(\phi_1(S_1))\le \frac{7}{4}$ and hence $\mathsf k(S_0\bdot S_1)\le \frac{1}{3}+\frac{7}{12}<1$. If $S_0\bdot S_1$ is not product-one free, then $S_0\bdot S_1$ has a nonempty product-one subsequence, which  is a tiny product-one subsequence,  a contradiction. Therefore  $S_0\bdot S_1$ is product-one free and 
	     let $U$ be the maximal subsequence of $S$ such that $S_0\bdot S_1$ divides $U$ and $U$ is product-one free.  It follows by $|U|\ge |S_0\bdot S_1|\ge 8$ and Theorem \ref{t3.2} that $U$ is smooth.
	      Assume to the contrary that $U\neq S$. Let $g\in \supp(S\bdot U^{[-1]})$. Then $U\bdot g$ is not product-one free. It follows by Lemma \ref{le1} that  $U\bdot g$ has a tiny product-one subsequence,  a contradiction.

	  Suppose  $|S_0|=0$.  Then $|S_1|\ge 8$. If $\mathsf k(\phi_1(S_1))\ge 3$, then we can choose three disjoint tiny product-one subsequences $\phi(W_1)$, $\phi(W_2)$, $\phi_3(W_3)$ of $\phi_1(S_1)$. 
	   Since $\mathsf k(W_i)\le \mathsf k(\sigma(W_i))$ for each $i\in [1,3]$ and $\sigma(W_1)\bdot \sigma(W_2)\bdot \sigma(W_3)$ is a sequence over $H_2$ with length $|\sigma(W_1)\bdot \sigma(W_2)\bdot \sigma(W_3)|=3\ge |H_2|$, we obtain $\sigma(W_1)\bdot \sigma(W_2)\bdot \sigma(W_3)$  and hence $ W_1\bdot W_2\bdot W_3$ have a tiny product-one subsequence, a contradiction.
 Otherwise $\mathsf k(\phi(S_1))<3$ and hence $\mathsf k(S_1)<1$.
 If $S_1$ is not product-one free, then $S_1$ has a nonempty product-one subsequence, which is a tiny product-one subsequence, a contradiction. Therefore  $S_1$ is product-one free and 
let $U$ be the maximal subsequence of $S$ such that $S_1$ divides $U$ and $U$ is product-one free.  It follows by $|U|\ge | S_1|\ge 8$ and Theorem \ref{t3.2} that $U$ is smooth.
Assume to the contrary that $U\neq S$. Let $g\in \supp(S\bdot U^{[-1]})$. Then $U\bdot g$ is not product-one free and hence by Lemma \ref{le1}  $U\bdot g$ has a tiny product-one subsequence,  a contradiction.
 \end{proof}

\begin{proposition}\label{c3.3}
	Let $G$ be a cyclic group of order $n\ge 3$  and let $S$ be a sequence of length  $|S|=n-1$. If $S$ has no tiny product-one subsequence, then $S=g^{[n-1]}$ for some $g\in G$ with $\ord(g)=n$. 	
\end{proposition}

\begin{proof}
Suppose $S$ has no tiny product-one subsequence. If $n\ge 10$, it follows by Proposition \ref{c3.2} that $S$ is smooth and hence there exists $g\in G$ with $\ord(g)=n$ such that $S=g^{[n-1]}$. If $n\in \{3, 4, 5, 7, 8, 9\}$, it follows by Theorem \ref{t3.1} that  $S$ is smooth and hence there exists $g\in G$ with $\ord(g)=n$ such that $S=g^{[n-1]}$.

  Now we deal with the only left case that  $n=6$. Suppose $G\cong H_1\oplus H_2$, where $H_1, H_2$ are  cyclic subgroups  with $|H_1|=2$ and $|H_2|=3$. Let $\phi_1\colon G\rightarrow H_1$ and $\phi_2\colon G\rightarrow H_2$ be projections from $G$ to $H_1$ and $H_2$. Then the sequence $S$ has a decomposition $$S=T_0\bdot S_0\bdot S_1\,,$$
where $S_0$ consists by terms $g$ with $\phi_1(g)=0$, $T_0$ consists by terms $g$ with $\phi_2(g)=0$, and $S_1$ consists by terms $g$ with $\phi_2(g)\neq 0$. 
	
	 If $|S_0|\ge 3=|H_2|$, then $S_0$ has a tiny product-one subsequence, a contradiction. If $|T_0|\ge 2=|H_1|$, then $T_0$ has a tiny product-one subsequence, a contradiction. Therefore $|S_0|\le 2$ and $|T_0|\le 1$, which implies that $|S_0\bdot S_1|\ge 4$. Let $W$ be a subsequence of $S_0\bdot S_1$  with $|W|=4$. Then $\mathsf k(W)\le \frac{2}{3}+\frac{2}{6}=1$. If $W$ is not product-one free, then $W$ has a nonempty product-one subsequence, which is  a tiny product-one subsequence, a contradiction. Otherwise $W$ is product-one free and let $U$ be the maximal subsequence of $S$ with $W\mid U$ such that $W$ is product-one free. It follows by $|U|\ge |W|\ge 4$ and Theorem \ref{t3.2} that $U$ is smooth. If $U=S$, then there exists $g\in G$ with $\ord(g)=6$ such that $S=U=g^{[5]}$. If $U\neq S$, let $g\in \supp(S\bdot U^{[-1]})$ and hence $U\bdot g$ is not product-one free. It follows by Lemma \ref{le1} that $U\bdot g$ has a  tiny product-one subsequence, a contradiction.
\end{proof}


\bigskip
\section{On Dihedral groups and Dicyclic groups} \label{3}
\bigskip

In this section, we deal with dihedral and dicyclic groups $G$ and consider both  direct and inverse problems of the invariants $\eta(G)$ and $\mathsf {ti}(G)$. 
We need the following two theorems, which  give  full answers of the direct and the inverse problems of the small Davenport constant for dihedral and dicyclic groups.

\begin{theorem} \label{th4.1} Let $G$ be a dihedral group of order $2n$ with $n\ge 3$. Then $\mathsf d(G)=n$ and if 
	 $S$ is  a product-one free  sequence over $G$ of length $n$, then one of the following holds.
	\begin{enumerate}
		\item There exist $\alpha, \tau\in G$ with $G=\langle \alpha, \tau \colon \alpha^n=\tau^2=1 \text{ and } \alpha\tau=\tau\alpha^{-1}\rangle$ such that $S=\alpha^{[n-1]}\bdot \tau$.
		\item $n=3$ and there exist $\alpha, \tau\in G$ with $G=\langle \alpha, \tau \colon \alpha^3=\tau^2=1 \text{ and } \alpha\tau=\tau\alpha^{-1}\rangle$ such that $S=\tau\bdot \alpha\tau\bdot \alpha^2\tau$.
	\end{enumerate}
\end{theorem}

\begin{proof}
	The assertion follows by \cite{Ol-Wh77} and \cite[Theorem 1.3]{Br-Ri18}.
	\end{proof}

\begin{theorem} \label{th4.2} Let $G$ be a dicyclic group of order $4n$ with $n\ge 2$. Then $\mathsf d(G)=2n$ and if 
	$S$ is  a product-one free  sequence over $G$ of length $2n$, then one of the following holds.
	\begin{enumerate}
		\item There exist $\alpha, \tau\in G$ with $G=\langle \alpha, \tau \colon \alpha^{2n}=1, \tau^2=\alpha^n, \text{ and } \alpha\tau=\tau\alpha^{-1}\rangle$ such that $S=\alpha^{[2n-1]}\bdot \tau$.
		\item $n=2$ and there exist $\alpha, \tau\in G$ with $G=\langle \alpha, \tau \colon \alpha^4=1, \tau^2=\alpha^2, \text{ and } \alpha\tau=\tau\alpha^{-1}\rangle$ such that either $S=\tau^{[3]}\bdot \alpha\tau$ or $S=\tau^{[3]}\bdot \alpha$.
	\end{enumerate}
\end{theorem}

\begin{proof}
	The assertion follows by \cite{Ol-Wh77} and \cite[Theorem 1.4]{Br-Ri18}.
	\end{proof}

\smallskip

Now we study the direct and the inverse problems of $\eta(G)$ for dihedral and dicyclic groups.

\smallskip
\begin{theorem} \label{t1}~
Let $G$ be a dihedral group of order $2n$ with $n\ge 3$. Then  $\eta (G) = n + 1$ and if $S$ is a sequence over $G$ with $|S|=n$ such that $S$ has no short product-one subsequence, then there exist $g,h\in G$ with $G=\langle g, h\, \colon\, g^n=h^2=1, hg=g^{-1}h\rangle$ such that $S=g^{[n-1]}\bdot h$ or $S=h\bdot gh\bdot g^2h$ (the latter case can only happen when $n=3$).
\end{theorem}

\begin{proof}
	Note that $\max\{\ord(g)\colon g\in G\}=n$. Thus a short sequence over $G$ is a sequence $S$ over $G$ with $1\le |S|\le n$.

1. Clearly Theorem \ref{th4.1} implies  $n+1 = \mathsf d (G) +1 \le \eta (G)$. Let $S \in \mathcal F (G)$ be a sequence of length $|S| = n+1$. We need to show that there exists a product-one subsequence of $S$ having length at most $n$. We may assume that $1\not\in \supp(S)$. Let $T \t S$ be a subsequence of length $|T| = n$.
If $T$ has a non-trivial product-one subsequence, then we are done. If $T$ has no non-trivial product-one subsequence, then it is product-one free sequence of length $\mathsf d (G) = n$. By Theorem \ref{th4.1}, we distinguish two cases.

Suppose $n=3$ and  there exist $\alpha,\tau\in G$ with $G=\langle \alpha,\tau\colon \alpha^3=\tau^2=1, \alpha\tau=\tau \alpha^{-1}\rangle$ such that  $T=\tau\bdot \alpha\tau\bdot \alpha^2\tau$.
If $S\bdot T^{[-1]}=\alpha^r$ for some $r\in [1,2]$, then $\alpha^r\bdot \tau\bdot \alpha^r\tau$ is a short product-one subsequence of $S$. If $S\bdot T^{[-1]}=\alpha^r\tau$ for some $r\in [0,2]$, then $\alpha^r\tau\bdot \alpha^r\tau$ is a short product-one subsequence of $S$.

Suppose there exist $\alpha,\tau\in G$ with $G=\langle\alpha,\tau \colon \alpha^n=\tau^2=1, \alpha\tau=\tau \alpha^{-1}\rangle$ such that  $T=\alpha^{[n-1]}\bdot\tau$. If $S\bdot T^{[-1]}=\alpha^r$ for some $r\in [1,n-1]$, then $\alpha^{r}\bdot \alpha^{[n-r]}$ is a short product-one subsequence. If $S\bdot T^{[-1]}=\alpha^r\tau$ for some $r\in [0,n-1]$,   then both $\alpha^r\tau\bdot \alpha^{[r]}\bdot \tau$ and $\alpha^{[n-r]}\bdot \alpha^r\tau\bdot \tau$ are product-one subsequences of $S$. Since $$|\alpha^r\tau\bdot \alpha^{[r]}\bdot \tau|+|\alpha^{[n-r]}\bdot \alpha^r\tau\bdot \tau|=1+r+1+n-r+1+1=n+4\le 2n+1\,,$$ we obtain one of them must be a short product-one subsequence.

\medskip
	2. 
Suppose $|S|=n$ and $S$ has no short product-one subsequence, which means $S$ is product-one free. The assertion follows by Theorem \ref{th4.1}.
\end{proof}


\smallskip
\begin{theorem} \label{t2}
	Let $G$ be a dicyclic group of order $4n$ with $n\ge 2$. Then
	$\eta (G) = 2n + 1$  and if $S$ is a sequence over $G$ with $|S|=2n$ such that $S$ has no short product-one subsequence, then there exist $g,h\in G$ with $G=\langle g,h\colon g^{2n}=1, h^2=g^n, hg=g^{-1}h\rangle$ such that $S=g^{[2n-1]}\bdot h$, or $S=g\bdot h^{[3]}$, or $S=gh\bdot h^{[3]}$ (the latter two cases can only happen when $n=2$). 
\end{theorem}

\begin{proof}	Note that $\max\{\ord(g)\colon g\in G\}=2n$. Thus a short sequence over $G$ is a sequence $S$ over $G$ with $1\le |S|\le 2n$.

	1. Clearly Theorem \ref{th4.2} implies that $2n + 1 = \mathsf d (G) + 1 \le \eta (G)$. Let $S \in \mathcal F (G)$ be a sequence of length $|S| = 2n+1$. We need to show that there exists a product-one subsequence of $S$ of length at most $2n$. We may assume that $1\not\in \supp(S)$. Let $T \t S$ be a subsequence of length $|T| = 2n$. If $T$ has a nonempty product-one subsequence, then we are done. If not, then it is product-one free sequence of length $\mathsf d (G) = 2n$. We distinguish three cases depending on Theorem \ref{th4.2}.
	
	Suppose $n=2$ and there exist $\alpha, \tau\in G$ with $G=\langle \alpha, \tau \colon \alpha^4=1, \tau^2=\alpha^2, \text{ and } \alpha\tau=\tau\alpha^{-1}\rangle$ such that  $T=\tau^{[3]}\bdot \alpha\tau$.
	If $S\bdot T^{[-1]}=\alpha^r$ for some $r\in [1,3]$, then one of the sequences  $\alpha^r\bdot \alpha\tau\bdot \tau, \alpha^r\bdot \tau^{[2]}, \alpha^r\bdot \tau\bdot \alpha\tau$ must be a short product-one subsequence of $S$. If $S\bdot T^{[-1]}=\alpha^r\tau$ for some $r\in [0,3]$, then one of the sequences  $\alpha^r\tau\bdot \tau^{[3]}, \alpha^r\tau\bdot \alpha\tau\bdot \tau^{[2]}, \alpha^r\tau\bdot \tau, \alpha^r\tau\bdot \alpha\tau$ must be a short product-one subsequence of $S$.

	Suppose $n=2$ and there exist $\alpha, \tau\in G$ with $G=\langle \alpha, \tau \colon \alpha^4=1, \tau^2=\alpha^2, \text{ and } \alpha\tau=\tau\alpha^{-1}\rangle$ such that $T=\tau^{[3]}\bdot \alpha$.
	If $S\bdot T^{[-1]}=\alpha^r$ for some $r\in [1,3]$, then one of the sequences  $\alpha^r\bdot \alpha\bdot \tau^{[2]}, \alpha^r\bdot \tau^{[2]}, \alpha^r\bdot  \alpha$ must be a short product-one subsequence of $S$. If $S\bdot T^{[-1]}=\alpha^r\tau$ for some $r\in [0,3]$, then one of the sequences  $\alpha^r\tau\bdot \tau^{[3]}, \alpha\bdot \alpha^r\tau\bdot \tau, \alpha^r\tau\bdot \tau, \alpha^r\tau\bdot \alpha\bdot \tau$ must be a short product-one subsequence of $S$. 
	
	Suppose there exist $\alpha, \tau\in G$ with $G=\langle \alpha, \tau \colon \alpha^{2n}=1, \tau^2=\alpha^n, \text{ and } \alpha\tau=\tau\alpha^{-1}\rangle$ such that $S=\alpha^{[2n-1]}\bdot \tau$. If $S\bdot T^{[-1]}=\alpha^x$ for some $x\in [1, 2n-1]$, then $\alpha^x\bdot \alpha^{[2n-x]}$ is a short product-one subsequence of $S$.	
Suppose  $S\bdot T^{[-1]} =\alpha^{y}\tau$ for some $y \in [0,2n-1]$. If $y\le n$, then $\alpha^{[n-y]}\bdot \alpha^y\tau\bdot \tau$ is a short product-one subsequence. If $y> n$, then $\alpha^y\tau\bdot \alpha^{y-n}\bdot \tau$ is a short product-one subsequence.
	
	\medskip
	2. 	
Suppose $S$ is a sequence over $G$ of length $2n$.
	If $S$ has no short product-one subsequence, then $S$ is product-one free. The assertion follows by Theorem \ref{th4.2}.
\end{proof}

\medskip
\begin{proof}[Proof of Theorem \ref{th2}]
 Suppose $\alpha,\tau\in G$ such that $G=\langle \alpha,\tau \colon \alpha^n=\tau^2=1, \alpha\tau=\tau\alpha^{-1}\rangle$. Let $H=\langle \alpha\rangle$ and $G_0=G\setminus H$. Set $W=\alpha^{[n-1]}\bdot \tau\bdot \tau\alpha\bdot \ldots \bdot \tau\alpha^{n-1}$.  Then $|W|=2n-1$. If $V$ is a nonempty product-one subsequence of $W$, then $|V_{G_0}|\ge 2$ and $V\neq V_{G_0}$. Thus $\mathsf k(V)>\mathsf k(V_{G_0})\ge 1$ and hence $W$ has no tiny product-one subsequence. Therefore $\mathsf {ti}(G)\ge 2n$.

1.  	Let $S$ be a sequence over $G$ of length $2n$. We only need to show that $S$ has a tiny product-one subsequence.
	 If $|S_{G_0}|\ge n+1$, then there exists $j\in [0,n-1]$ such that $(\alpha^j\tau)^{[2]}\t S_{G_0}$, which implies $(\alpha^j\tau)^{[2]}$ is a tiny product-one subsequence of $S$. Otherwise $|S_{H}|\ge n=|H|$. Then $S_H$ and hence $S$ have a tiny product-one subsequence.

	\medskip
	2. 
	Suppose $|S|=2n-1$ and $S$ has no tiny product-one subsequence. We only need to show that $S$ has the desired form.
	If $\mathsf h(S_{G_0})\ge 2$, then there exists $g\in G_0$ such that $\mathsf v_g(S_{G_0})\ge 2$, whence $g^{[2]}$ is a tiny product-one subsequence of $S$, a contradiction. Therefore $S_{G_0}$ is square free and hence $|S_{G_0}|\le n$.
	 If $|S_H|\ge n=|H|$, then $S_H$ and hence $S$ have a tiny product-one subsequence, a contradiction. Therefore $|S_H|=n-1$ and $|S_{G_0}|=n$. The assertion follows by Proposition \ref{c3.3} and $S_{G_0}$ is square free.
	\end{proof}

\begin{lemma}\label{l5.2}
	 Let $G=\langle\alpha, \tau\colon \alpha^{2n}=1, \tau^2=\alpha^n, \text{ and } \alpha\tau=\tau\alpha^{-1} \rangle$ be a dicyclic group of order $4n$ with $n\ge 116$,  let $H=\langle\alpha\rangle$, let $G_0=G\setminus H$,  and let $S$ be a sequence over $G$ with $|S|=2n$ and $|S_{G_0}|\ge 2$. Then 
	 $S$ has a tiny product-one subsequence
\end{lemma}

\begin{proof}
	  Assume to the contrary that $S$ has no tiny product-one subsequence. We distinguish two cases.
	
	\medskip
	\noindent
	{\bf CASE 1.} \, $|S_H| \ge \frac{9n}{5}$.
	
	Since $S_H$ has no tiny product-one subsequence and $|S_H|\ge \frac{9|H|}{10}$, it follows by Proposition \ref{c3.2} that 
	$S_H$ is $g$-smooth for some $g\in H$ with $\ord(g)=2n$. Then $$\{g, g^2, \ldots, g^{\lceil \frac{9n}{5}\rceil}\}\subset\Pi(S_H)\,.$$
	 Note that $G_0=\{g^r\tau\colon r\in [0, 2n-1]\}$.
	Choose $g^i\tau, g^j\tau\in \supp(S_{G_0})$ with $0\le i\le j\le 2n-1$ such that $g^i\tau\bdot g^j\tau\mid S_{G_0}$. Then $j-i\le n$ or $2n+i-j\le n$. 
	
	Suppose $j-i\le n$. 
	Let $T$ be a subsequence of $S_H$ such that $\sigma(T)=g^{j-i}$. Then $T\bdot g^i\tau\bdot g^j\tau$ is a product-one subsequence and it follows by Lemma \ref{le1} that $\mathsf k(T\bdot g^i\tau\bdot g^j\tau)= \frac{2}{4}+\mathsf k(T)\le \frac{2}{4}+\frac{j-i}{2n}\le 1$, whence $T\bdot g^i\tau\bdot g^j\tau$ is a tiny product-one sequence, a contradiction.

	Suppose $2n+i-j\le n$. 
	Let $T$ be a subsequence of $S_H$ such that $\sigma(T)=g^{2n+i-j}$. Then $T\bdot g^i\tau\bdot g^j\tau$ is a product-one subsequence and it follows by Lemma \ref{le1} that $\mathsf k(T\bdot g^i\tau\bdot g^j\tau)= \frac{2}{4}+\mathsf k(T)\le \frac{2}{4}+\frac{2n+i-j}{2n}\le 1$, whence $T\bdot g^i\tau\bdot g^j\tau$ is a tiny product-one sequence, a contradiction

	\medskip
	\noindent
	{\bf CASE 2.} \, $|S_{G_0}| \ge \lceil\frac{n}{5}\rceil+1$.
	
	Suppose $S_{G_0}=\alpha^{r_1}\tau\bdot \ldots \alpha^{r_{|S_{G_0}|}}\tau$, where $r_1,\ldots, r_{|S_{G_0}|}\in [0, 2n-1]$. Since $S_{G_0}$ has no product-one subsequence of length $4$, we obtain
	 $r_i+r_j\not\equiv r_s+r_t\pmod {2n}$ for any distinct $i,j,s,t\in [1, |S_{G_0}|]$. Let $W=r_1\bdot \ldots \bdot r_{|S_{G_0}|}$. Then $\mathsf h(W)\le 3$ and there are at most one term of $W$ with the multiplicity larger than $1$, whence
	by taking at most two terms out, the remaining sequence $W'$ of $W$ is square free. Since $${{|W'|}\choose2}\ge {{|S_{G_0}|-2}\choose2}\ge \frac{(\lceil0.2n\rceil-1)( 0.2n -2)}{2}\ge 0.1n\lceil0.2n\rceil-0.3n > 2n=|G_0|\,,$$
	there exist $i,j,s,t\in [1, |S_{G_0}|]$ with $i\neq j$, $s\neq t$, $\{i,j\}\neq \{s,t\}$, $r_i\bdot r_j$ divides $W'$, and $r_s\bdot r_t$ divides $W'$, such that $r_i+r_j\equiv r_s+r_t\pmod {2n}$. Therefore $\{i,j\}\cap \{s,t\}\neq \emptyset$, which implies that $\{i,j\}=\{s,t\}$, a contradiction.
\end{proof}


\begin{proof}[Proof of Theorem \ref{th3}]
	 Suppose $n\ge 116$ and suppose there exist $g,h\in G$ such that $G=\langle g, h\colon  g^{2n}=1, h^2=g^n, \text{ and } hg=g^{-1}h\rangle$. Let $W=g^{[2n-1]}\bdot h$. Then it is easy to see that $W$ has no tiny product-one subsequence. Thus $\mathsf {ti}(G)\ge 2n+1$.
	 
	 1. Let $S$ be a sequence of length $|S| = 2n+1$. We only need to show that there exists a tiny product-one subsequence of $S$. If $|S_H|\ge 2n=|H|$, then $S$ has a tiny product-one sequence. Otherwise $|S_{G_0}|\ge 2$.
	It follows by Lemma \ref{l5.2} that $S$ has a tiny product-one subsequence.

	\medskip
	2. Suppose $|S|=2n$ and $S$ has no tiny product-one subsequence. We only need to show that $S$ has the desired form.
	If $|S_H|=2n=|H|$, then $S$ has a tiny product-one sequence, a contradiction.
	If  $|S_{G_0}|\ge 2$,  it follows by Lemma \ref{l5.2} that $S$ has a tiny product-one subsequence, a contradiction.
	Therefore $|S_{G_0}|=1$ and $|S_H|=2n-1$. Since $S_H$ has no tiny product-one subsequence, the assertion follows by Proposition \ref{c3.3}.	
\end{proof}


\providecommand{\bysame}{\leavevmode\hbox to3em{\hrulefill}\thinspace}
\providecommand{\MR}{\relax\ifhmode\unskip\space\fi MR }
\providecommand{\MRhref}[2]{%
  \href{http://www.ams.org/mathscinet-getitem?mr=#1}{#2}
}
\providecommand{\href}[2]{#2}

\end{document}